\documentclass[12pt]{amsart}
\usepackage{fullpage}
\usepackage{amsmath}
\usepackage{amssymb}
\usepackage{verbatim}
\usepackage[all]{xy}
\usepackage{color}
\usepackage{graphicx}
\usepackage{tikz} 
\usetikzlibrary{matrix,arrows} 
\def\p{\partial}
\newtheorem{Theorem}{Theorem}[section]
\newtheorem{Definition}{Definition}[section]

\newtheorem{Lemma}[Theorem]{Lemma}

\newtheorem{Corollary}[Theorem]{Corollary}

\newtheorem{Remark}{Remark}[section]
\numberwithin{equation}{section}
\newcommand{\oD}{ \!{\buildrel \circ 
\over D}^{_{_{_{\mbox{{\small $_{q+1}$}}}}}}}

\newcommand{\m}{\mathfrak{m}}

\newcommand{\gtor}{\bar{g}_{tor}}

\def\N{{\mathcal N}}

\def\Riem{{\mathcal R}{\mathrm i}{\mathrm e}{\mathrm m}}

\def\GL{{\mathcal G}{\mathcal L}}

\def\m{\mathfrak m}
\def\min{\mathrm{m}\mathrm{i}\mathrm{n}}

\begin{document}
\author{Mark Walsh}
\title{Cobordism Invariance of the homotopy type of the space of positive scalar curvature metrics}
\begin{abstract}
We show that the homotopy type of the space of metrics of positive scalar curvature on a smooth manifold remains unchanged, after application of surgery in codimension at least three to the underlying manifold. This result is originally due to V. Chernysh, but remains unpublished. 
\end{abstract}

\maketitle
\section{Introduction}\label{intro}
This paper concerns the space of metrics of positive scalar curvature ({psc-metrics}) on a smooth manifold $X$. Denoted $\Riem^{+}(X)$, this space is an open subspace of the space of all Riemannian metrics on $X$, $\Riem(X)$, equipped with its standard smooth topology. In general, little is known about the topology of $\Riem^{+}(X)$, although some results have been obtained at the level of $0$ and $1$-connectivity; see \cite{Walsh2} for a survey. This is in contrast to the problem of whether or not $X$ admits a psc-metric, of which a great deal is known; see \cite{RS}.

We are interested in the homotopy type of $\Riem^{+}(X)$ and how it is affected by surgery on the underlying manifold. Our main result, Theorem \ref{Chernyshthm}, is as follows. 
\vspace{0.2cm}

\noindent {\bf Main Theorem.}
{\em Let $X$ be a smooth compact manifold of dimension $n$. Suppose $Y$ is obtained from $X$ by surgery on a sphere $i:S^{p}\hookrightarrow X$ with $p+q+1=n$ and  $p,q\geq 2$. Then the spaces $\Riem^{+}(X)$ and $\Riem^{+}(Y)$ are homotopy equivalent.}
\vspace{0.2cm}

\noindent A consequence of this, Corollary \ref{spinChernysh}, is that in the case of simply connected spin manifolds of dimension at least five, the homotopy type of the space of psc-metrics is a spin-cobordism invariant. The main result of this paper is originally due to V. Chernysh in \cite{Che}, although remains unpublished. Our proof is much shorter and makes use of work done in \cite{Walsh1}.

The main idea behind the proof of Theorem \ref{Chernyshthm}, is to exhibit a homotopy equivalence between the space $\Riem^{+}(X)$ and a certain subspace ${\Riem}_{std}^{+}(X)$. This subspace consists of metrics which are ``standard" near an embedded surgery sphere. If $Y$ is obtained from $X$ by surgery on this sphere, then $X$ is obtainable from $Y$ by a complementary surgery. In turn, $\Riem^{+}(Y)$ is homotopy equivalent to a subspace ${\Riem}_{std}^{+}(Y)$, of metrics which are standard near this complementary surgery sphere. The space $\Riem_{std}^{+}(Y)$ is demonstrably homotopy equivalent to ${\Riem}_{std}^{+}(X)$. To show that $\Riem^{+}(X)$ and ${\Riem}_{std}^{+}(X)$ are homotopy equivalent, we use an intermediary space of ``almost standard" metrics, ${\Riem}_{Astd}^{+}(X)$, where ${\Riem}_{std}^{+}(X)\subset {\Riem}_{Astd}^{+}(X)\subset \Riem^{+}(X)$. In Lemma \ref{Chetech}, we show that there is a deformation retract from ${\Riem}_{Astd}^{+}(X)$ to ${\Riem}_{std}^{+}(X)$. It remains to show homotopy equivalence between ${\Riem}_{Astd}^{+}(X)$ and ${\Riem}^{+}(X)$. We now use the fact, due to Palais in \cite{Palais}, that both of these spaces are dominated by CW-complexes. Thus, by a theorem of Whitehead, it is enough to show that the relative homotopy groups $\pi_{k}(\Riem^{+}(X),{\Riem}_{Astd}^{+}(X))$ are trivial for all $k$. This is the most delicate step and is achieved by a family version of the surgery technique of Gromov and Lawson; see \cite{Walsh1} for a detailed account. A crucial requirement of this technique is that the embedded surgery spheres are in codimension at least three, hence the need for this hypothesis.

My thanks to Boris Botvinnik at the University of Oregon, Christine Escher at Oregon State University and David Wraith at NUI Maynooth, Ireland, for their helpful comments. 

\section{Background}
Given its importance in our work, it is worth recalling what we mean by surgery.
Let $X$ denote a smooth manifold of dimension $n$ and let $i:S^{p}\hookrightarrow X$ be an embedding with trivial normal bundle. Thus, we can extend $i$ to an embedding $\bar{i}:S^{p}\times D^{q+1}\hookrightarrow X$, where $p+q+1=n$. By removing an open neighbourhood of $S^{p}$, we obtain a manifold $X\setminus\bar{i}(S^{p}\times {\oD})$ with boundary $\bar{i}(S^{p}\times S^{q})$. Here ${\oD}$ denotes the interior of the disk $D^{q+1}$. As the handle $D^{p+1}\times S^{q}$ has diffeomorphic boundary, we can use the map $\bar{i}|_{S^{p}\times S^{q}}$, to glue the manifolds $X\setminus \bar{i}(S^{p}\times {\oD})$ and $D^{p+1}\times S^{q}$ together by identifying their boundaries and obtain the manifold 
\begin{equation*}
Y=(X\setminus\bar{i}(S^{p}\times {\oD}))\cup_{\bar{i}}(D^{p+1}\times S^{q}).
\end{equation*}
The manifold $Y$ is said to be obtained from $X$ by {\em surgery} on the embedding $\bar{i}$ (or by a $p$-surgery or codimension $q+1$-surgery). In particular, this surgery may be ``reversed" by performing a {\em complementary surgery} on the above embedding $D^{p+1}\times S^{q}\hookrightarrow Y$, to restore the original manifold $X$.

The principal technique in the construction of new psc-metrics comes from the Surgery Theorem of Gromov and Lawson in \cite{GL} (proved independently by Schoen and Yau \cite{SY}). 

\vspace{0.3cm}
\noindent{\bf Surgery Theorem.} {\em Let $X$ be a smooth manifold. If $X$ admits a metric of positive scalar curvature, then so does any manifold which is obtained from $X$ by surgery in codimension at least three.}
\vspace{0.3cm}

\noindent A consequence of this theorem was an enormous increase in the number of known examples of manifolds which admit psc-metrics. Moreover, the proof by Gromov and Lawson is constructive. More precisely, given a psc-metric $g$ on $X$, there is an explicit technique for building a new psc-metric $g'$ on $Y$, the manifold obtained by surgery. Shortly, we will provide a brief review of this construction. Before doing this, we need to introduce some notation. 

We denote, by $ds_n^{2}$,  the standard round metric of radius $1$ on the sphere $S^{n}$. Let $f:[0, \infty)\rightarrow[0, \infty)$ be a smooth function which satisfies the following conditions.
\begin{enumerate}
\item[{\bf (i)}] $f(0)=0$, $f'(0)=1$.
\item[{\bf (ii)}] The $m^{th}$ derivative $f^{(m)}(0)=0$, when $m$ is even.
\item[{\bf (iii)}] $f(t)>0$, when $t>0$.
\end{enumerate}
Then, the metric $dr^{2}+f(r)^{2}ds_{n-1}^{2}$ on $(0, \infty)\times S^{n-1}$ extends uniquely to a smooth metric on the plane $\mathbb{R}^{n}$, where $r$ denotes the radial distance coordinate. This follows from the results of Chapter 1, Section 3.4 of \cite{P}. Moreover, the resulting ``warped product metric" is radially symmetric and has scalar curvature $R$ given by the following formula.
\begin{equation}\label{Rcurv}
R=-2(n-1)\frac{f''}{f}+(n-1)(n-2)\frac{1-(f')^{2}}{f^{2}}.
\end{equation}

A important example of such a metric is determined as follows.
Let $f_1:[0, \infty)\rightarrow[0, \infty)$ be a smooth function which satisfies the following conditions.
\begin{enumerate}
\item[{\bf (i)}] $f_1(t)=\sin{t}$, when $t$ is near $0$.
\item[{\bf (ii)}] $f_1(t)=1$, when $t\geq\frac{\pi}{2}$.
\item[{\bf (iii)}] ${f_{1}}''(t)< 0$, when $0\leq t<\frac{\pi}{2}$.
\end{enumerate}
More generally, for each $\delta>0$, the function $f_\delta:[0,\infty)\rightarrow[0,\infty)$ is defined by the formula
\begin{equation*}
f_{\delta}(t)=\delta f_1 (\frac{t}{\delta}).
\end{equation*}
By restricting $f_{\delta}$ to the interval $(0,b]$, where $b\geq\delta\frac{\pi}{2}$, the metric $dt^{2}+f_{\delta}(t)^{2}ds_{n-1}^{2}$ on $(0,b]\times S^{n-1}$, extends uniquely to a smooth $O(n)$-symmetric metric on the disk $D^{n}$. This metric, known as a {\em torpedo metric}, is denoted $\gtor^{n}(\delta)$; see Fig. \ref{torpedo}. It is a round $n$-sphere of radius $\delta$ near the centre of the disk, and a standard product of $(n-1)$-spheres of radius $\delta$ near the boundary, at least when $b>\delta\frac{\pi}{2}$. In the case when $b=\delta\frac{\pi}{2}$, the metric is only a product near the boundary infinitesimally, though it smoothly attaches to a standard product $dt^{2}+\delta ds_{n-1}^{2}$. This is sufficient for our purposes.  From equation \ref{Rcurv}, we see that the scalar curvature of such a metric can be bounded below by an arbitrarily large positive constant, by choosing sufficiently small $\delta$.  
\begin{figure}[!htbp]
\vspace{-2cm}
\begin{picture}(0,0)%
\includegraphics{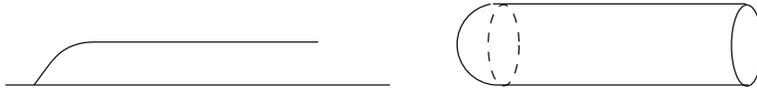}%
\end{picture}%
\setlength{\unitlength}{3947sp}%
\begingroup\makeatletter\ifx\SetFigFont\undefined%
\gdef\SetFigFont#1#2#3#4#5{%
  \reset@font\fontsize{#1}{#2pt}%
  \fontfamily{#3}\fontseries{#4}\fontshape{#5}%
  \selectfont}%
\fi\endgroup%
\begin{picture}(5079,1559)(1902,-7227)
\end{picture}%
\caption{The function $f_\delta$ and the resulting torpedo metric}
\label{torpedo}
\end{figure}   

\begin{Remark} We have made a slight generalisation to the definition of torpedo metric given in \cite{Walsh1}, to allow for an infinitesimal product near the boundary. As we can still smoothly attach a standard product to this boundary, all of the results in \cite{Walsh1} still hold. In general, we will suppress the $\delta$ term when writing $\gtor^{n}(\delta)$ and simply write $\gtor^{n}$, knowing that we may choose $\delta$ to be arbitrarily small if necessary. 
\end{Remark}

We will now describe the Gromov-Lawson surgery technique. The set-up described here will be used throughout the rest of the paper. Fix a Riemannian metric $\m$ on $X$ and an embedding $i_0:S^{p}\hookrightarrow X$.  The Riemannian metric $\m$ is a reference metric and is not required to have positive scalar curvature. We assume also that the embedding $i_0$ has trivial normal bundle, denoted by $\N$. Moreover, let $p+q+1=n$ with $q\geq 2$.
By choosing an orthonormal frame for $\N$ over $i_0(S^{p})$, we specify a bundle isomorphism $\phi:S^{p}\times \mathbb{R}^{q+1}\rightarrow\mathcal{N}$. Points in $S^{p}\times \mathbb{R}^{q+1}$ will be denoted $(y,x)$. Let $r$ denote the standard Euclidean radial distance function in $\mathbb{R}^{q+1}$ and let $D^{q+1}({\rho})=\{x\in\mathbb{R}^{q+1}:r(x)\leq{\rho}\}$ denote the standard Euclidean disk of radius ${\rho}$. 
We will choose $\bar{\rho}>0$ sufficiently small, so that the composition $i_\rho=\exp\circ\phi|_{S^{p}\times D^{q+1}({\rho})}$, where $\exp$ denotes the exponential map with respect to the metric $\m$, is an embedding for all $\rho\in (0,\bar{\rho}]$. We will also denote by $N({\rho})$, the tubular neighbourhood of $i_0 (S^{p})$, which is the image of each such embedding. This information is summarised in the commutative diagram below, where $p_1$ denotes projection onto the first factor. 

\hspace{2cm}
\vspace{0.5cm}
\begin{tikzpicture}[description/.style={fill=white,inner sep=2pt}] 
\matrix (m) [matrix of math nodes, row sep=3em, 
column sep=2.5em, text height=1.5ex, text depth=0.25ex] 
{ S^{p}\times D^{q+1}({\rho}) & & S^{p}\times\mathbb{R}^{q+1} & & \N \\ 
&  & S^{p}  & & X\\ }; 
\path[->,font=\scriptsize] 
(m-1-3) edge node[auto] {$ \phi $} (m-1-5) 
(m-1-5) edge node[auto] {$\exp $} (m-2-5) 
(m-1-1) edge node[auto] {$ p_1 $} (m-2-3) 
(m-1-3) edge node[auto] {$ p_1 $} (m-2-3); 
\path[right hook->,font=\scriptsize] 
(m-1-1) edge node[auto]  {}(m-1-3) 
(m-2-3) edge node[auto]  {$i_0$}(m-2-5); 
\end{tikzpicture} 

Now, let $g$ be any psc-metric on $X$. In their proof, Gromov and Lawson show that  $g$ can be replaced by a psc-metric $g_{std}$, which satisfies the following properties.

\begin{enumerate}
\item{}$g_{std}=g$ on $X\setminus N(\bar{\rho})$.
\item{} For some $\rho_0\in (0,\bar{\rho})$ and some $\delta>0$, $g_{std}=ds_p^{2}+g_{tor}^{q+1}(\delta)$ on $N(\rho_0)$. 
\end{enumerate}
From here, the Surgery Theorem follows easily, as the standard part of $g_{std}$ can easily be replaced by a metric of the form $g_{tor}^{p+1}+\delta^{2}ds_{q}^{2}$.
A detailed analysis of this construction is performed in \cite{Walsh1} and so we will be as brief as possible. We will consider the construction as consisting of two stages.
\\

\noindent {\bf Stage 1.} We begin by constructing a particular hypersurface $M$ in $N(\bar{\rho})\times \mathbb{R}$, where $N(\bar{\rho})\times\mathbb{R}$ is equipped with the metric $g+dt^{2}$. The hypersurface is obtained by pushing out bundles of geodesic spheres of radius $r$ in $N(\bar{\rho})$, with respect to a smooth curve $\gamma$ in the $t-r$-plane of the type depicted in the first image in Fig. \ref{glcurves} below. This curve consists of three straight line segments: vertical, tilted and horizontal, which are connected by a pair of concave upward curves. Finally, the horizontal piece is connected to the $t$-axis by means of a curve of downward concavity which intersects the axis as a circular arc. Given such a curve $\gamma$, we denote by $g_{\gamma}$, the metric obtained by replacing $g$ with the induced hypersurface metric. The vertical segment of $\gamma$ allows for a smooth transition to the original metric on the rest of $X$. The fact that such a curve can be constructed to ensure positive scalar curvature of the induced metric is proved in \cite{GL} (although a minor error is later corrected in \cite{RS}). Of crucial importance is the fact that there is a sphere factor of dimension at least two, the curvature of which may be made arbitrarily large. We will refer to such curves as Gromov-Lawson curves. In \cite{Walsh1}, we prove the following stronger statement. 

\begin{Lemma}{\cite{Walsh1}}\label{glcurvelemma}
For each Gromov-Lawson curve $\gamma$, there is a homotopy $\gamma_{s}, s\in I$, through smooth curves in the $t-r$-plane, which satisfies the following conditions.
\begin{enumerate}
\item{} $\gamma_0$ is the segment $[0, \bar{\rho}]$ on the vertical axis.
\item{} Each curve $\gamma_s$ begins at the point $(0, \bar{\rho})$, proceeds initially as a vertical segment and finishes as a curve which, infinitesimally at least, intersects the $t$-axis as the arc of a circle.
\item{} $\gamma_1=\gamma$
\item{} The induced metric $g_s=g_{\gamma_{s}}$, has positive scalar curvature for all $s\in I$.
\end{enumerate}
\end{Lemma}  
\begin{proof}
This is done in great detail in \cite{Walsh1}. Roughly speaking, the homotopy $\gamma_s, s\in I$, is constructed by first performing a small tilt to the horizontal segment of $\gamma$, producing a curve of the type shown in the second picture of Fig. \ref{glcurves}. Then, treating this curve as the graph of a function $t=t(r)$ over the $r$ axis, a straightforward linear homotopy to the $r$-axis does the rest. Positivity of the scalar curvature is preserved by the fact that the linear homotopy decreases the second derivative of this curve but does not increase the radii of the geodesic spheres.
\end{proof}

\begin{figure}[!htbp]
\hspace{2cm}
\begin{picture}(0,0)%
\includegraphics{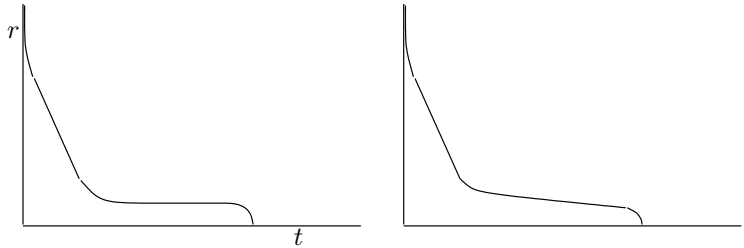}%
\end{picture}%
\setlength{\unitlength}{3947sp}%
\begingroup\makeatletter\ifx\SetFigFont\undefined%
\gdef\SetFigFont#1#2#3#4#5{%
  \reset@font\fontsize{#1}{#2pt}%
  \fontfamily{#3}\fontseries{#4}\fontshape{#5}%
  \selectfont}%
\fi\endgroup%
\begin{picture}(5079,1559)(1902,-7227)
\put(3614,-7336){\makebox(0,0)[lb]{\smash{{\SetFigFont{10}{8}{\rmdefault}{\mddefault}{\updefault}{\color[rgb]{0,0,0}$t$}%
}}}}
\put(1814,-6036){\makebox(0,0)[lb]{\smash{{\SetFigFont{10}{8}{\rmdefault}{\mddefault}{\updefault}{\color[rgb]{0,0,0}$r$}%
}}}}
\end{picture}%
\caption{Gromov-Lawson curves}
\label{glcurves}
\end{figure}   

\noindent Henceforth, we will use the term {\em isotopy} to describe a path in $\Riem^{+}{(X)}$. Also, metrics which lie in the same path component of $\Riem^{+}(X)$ will be said to be {\em isotopic}.
\\

\noindent{\bf Stage 2.}
Near ${i_0}S^{p}$, the metric $g_\gamma$ is approximately a Riemannian submersion metric on the total space of the bundle $\N(\bar{\rho})\rightarrow i_0(S^{p})$. The base metric is the induced metric on $i_0(S^{p})$. The metric on fibres is (due to the shape of $\gamma$) close to the standard torpedo metric $g_{tor}^{q+1}(\delta)$, where $\delta$ may need to be very small. A straightforward linear homotopy allows for adjustment of the fibres near $i_0(S^{p})$ to obtain a Riemannian submersion. Finally, using the formulae of O'Neill (Chapter 9 of \cite{B}), the positivity of the curvature on the disk factor allows us to homotopy through psc-submersion metrics, near $S^{p}$, to obtain the desired metric $g_{std}$. In \cite{Walsh1}, we actually prove something a little stronger.

\begin{Lemma}{\cite{Walsh1}}\label{GLisotopy}
There is an isotopy $\alpha:I=[0,1]\rightarrow\Riem^{+}(X)$ with $\alpha(0) =g$, $\alpha(1)=g_{std}$ and so that $\alpha(t)$ is a psc-metric for all $t\in I$.
\end{Lemma}
\begin{proof}
We combine the isotopy constructed in Lemma \ref{glcurvelemma} with a further isotopy through submersion metrics. This further isotopy involves a linear homotopy on fibres to obtain the desired torpedo metric on the disk factor, a linear homotopy on the base to obtain the desired round metric on $i_0{(S^{p})}$ and a linear homotopy through horizontal distributions to obtain a flat distribution and consequently a product metric. All of this is made possible by the fact that the metric on the fibres may be adjusted to carry arbitrarily large positive scalar curvature. As before, full details can be found in \cite{Walsh1}.
\end{proof} 

\section{Some observations about the Gromov-Lawson construction}
We now make a number observations, which will be of crucial importance in proving the main theorem. Firstly, for each $g\in\Riem^{+}(X)$, the isotopy $g_s, s\in I$, constructed in Lemma \ref{GLisotopy} is not unique. There are a number of choices involved in the construction, and so it is worth analysing how these choices impact the resulting isotopy. The choice of reference metric and the normal coordinate neighbourhood around the embedded surgery sphere are fixed and do not change as we vary the choice of psc-metric $g$. However, it is worth noting that choosing a different reference metric or a different set of normal coordinates has no effect on the process. We make this observation in the form of a lemma.

\begin{Lemma}\label{GLequiv}
The Gromov-Lawson isotopy $\alpha$ constructed in Lemma \ref{GLisotopy} is unaffected by the choice of reference metric $\m$ and bundle isomorphism $\phi:S^{p}\times \mathbb{R}^{q+1}\rightarrow\mathcal{N}$.
\end{Lemma}
\begin{proof}
In stage 1, all of our adjustments are radially symmetric and depend only on the Euclidean distance function $r$. This function is the same for any metric and so the choice of reference metric $\m$ is unimportant.  In stage 2, adjustments consist of linear homotopies to metrics which, again, are standard when $r$ is near zero and transition to the original metric by a smooth adjustment in the radial direction. Once again, the choice of $\m$ is unimportant. 

The choice of bundle isomorphism $\phi$, is a choice of orthonormal frame on the embedded surgery sphere. In particular, we can equate all such choices with the obvious action of $O(p)\times O(q+1)$ on this bundle. In Lemma 2.10 of \cite{Walsh2}, we show that the entire Gromov-Lawson construction is equivariant with respect to this action and so the choice of normal coordinates is unimportant. 
\end{proof}

There are some choices which do make a difference. For example, there are many possible choices of Gromov-Lawson curve $\gamma$ which will work. Furthermore, the radius of the $q$-dimensional sphere factor can be chosen to be arbitrarily small. However, these choices only mildly affect the isotopy. Overall, the construction does not involve any significant choices and is mostly a sequence of linear homotopies from arbitrary to standard pieces. With this in mind we make the following definition.

\begin{Definition}
{\em For each metric $g\in \Riem^{+}(X)$, isotopies of the type described in Lemma \ref{GLisotopy} are known as {\em Gromov-Lawson isotopies}}.
\end{Definition}
 
Before making our second observation, we define some important subspaces of $\Riem^{+}(X)$. 
Firstly, for each metric $g\in \Riem^{+}{(X)}$, let $\GL(g)$ be the subspace of $\Riem^{+}(X)$ obtained by taking the union of all metrics contained in all Gromov-Lawson isotopies emanating from $g$.
Next, we consider the space of psc-metrics on $X$ which are standard near $i_{0}(S^{p})$. This is

\begin{equation*}
\Riem_{std}^{+}(X)=\{g\in \Riem^{+}(X):i_\rho^{*}(g)=ds_p^{2}+g_{tor}^{q+1}{\text {\rm, for some }} \rho\in(0,\bar{\rho}] \}.
\end{equation*}

\noindent Finally, will also be interested in the space of all metrics obtained by Gromov-Lawson isotopy of these metrics. Thus, we define

\begin{equation*}
{\Riem}_{Astd}^{+}(X)=\bigcup_{g\in \Riem_{std}^{+}(X)} \GL(g) .
\end{equation*}

\noindent We now make our second observation in the form of a lemma.

\begin{Lemma}
The subspace ${\Riem}_{Astd}^{+}(X)\subset \Riem^{+}(X)$ remains fixed if we vary the choice of reference metric $\m$ or the bundle isomorphism $\phi$.
\end{Lemma}
\begin{proof}
This is an immediate consequence of Lemma \ref{GLequiv}.
\end{proof}

Our final observation, which also takes the form of a lemma, is the key technical step in proving our main theorem.
  
\begin{Lemma}\label{Chetech}
The spaces $\Riem_{std}^{+}(X)$ and ${\Riem}_{Astd}^{+}(X)$ homotopy equivalent.
\end{Lemma}
\begin{proof}
We will exhibit a deformation retract of the space ${\Riem}_{Astd}^{+}(X)$ down to the subspace $\Riem_{std}^{+}(X)$. More precisely, we will describe a process which continuously adjusts metrics in ${\Riem}_{Astd}^{+}(X)$, making them standard near $i_0{(S^{p})}$, while having no effect on metrics which are already in $\Riem_{std}^{+}(X)$. Before doing this we need to analyse the effect of the Gromov-Lawson isotopy on metrics in $\Riem_{std}^{+}(X)$.

Let $g\in \Riem_{std}^{+}(X)$. Let $\rho_{std}\in [0,\bar{\rho}]$ be the supremum over all $\rho\in   [0,\bar{\rho}]$, for which the metric $g$ takes the form 
\begin{equation}\label{radsym}
g|_{N(\rho_{std})}=ds_{p}^{2}+dr^{2}+w(r)^{2}ds_{q}^{2},
\end{equation}
on $N(\rho)$, where $w:(0, \bar{\rho})\rightarrow (0,\infty)$ is a smooth function. Note that, when $r$ is near $0$, the function $w$ is the torpedo function $f_{\delta}$ described earlier. We now restrict our attention to the region $N(\rho_{std})$. Applying stage 1 of the Gromov-Lawson construction to the metric $g$ produces an isotopy through metrics, which at all times takes the form shown in equation \ref{radsym}. Furthermore, although the function $w$ may not always be a torpedo function near $0$, it will at all times satisfy the following conditions.
\begin{enumerate}
\item{} $w'(0)=1$ and $w'(r)\geq 0$, for all $r\in(0, \rho_{std})$.
\item{} $w^{(m)}(0)=0$, when $m$ is even.
\item{} $w''(r)\leq 0$, when $r$ is near $0$.
\end{enumerate}
\noindent Finally, as each such metric is a standard product near $i_{0}(S^{p})$ and as stage 2 of the Gromov-Lawson construction is essentially just a sequence of linear homotopies to the standard form, we may conclude that stage 2 leaves all of these metrics unchanged. Thus, for each $g\in {\Riem}_{Astd}^{+}(X)$, there are parameters $0<\rho_{std}\leq \bar{\rho}$, so that on $N(\rho_{std})$, $g$ takes the form described in equation \ref{radsym} for some smooth function $w:(0, \rho_{std})\rightarrow (0,\infty)$ (determined by $g$) which, although not necessarily torpedo, satisfies the above conditions. 

Considering elements of ${\Riem}_{Astd}^{+}(X)$ in this way, allows us to construct the necessary deformation retract by focussing on the corresponding $w$ functions. In the case where $w$ is torpedo near $0$, we wish to make no change. In all other situations we wish to make the function $w$ torpedo-like. Before considering the various cases, we need some notation. Let $\rho'_{0}$ denote the infimum over $\rho\in(0, \rho_{std}]$ for which $w'(\rho)=0$. Let $\rho''_{0}$ denote the infimum over $\rho\in(0, \rho_{std}]$ for which $w''(\rho)=0$. Finally let $\rho_0$ be defined as
\begin{equation*}
 \rho_{0}=\min\{\rho_{std}, \rho'_{0}, \rho''_{0} \}.
\end{equation*}
The function $w$ (and consequently the parameter $\rho_{0}$) varies continuously over the space of metrics ${\Riem}_{Astd}^{+}(X)$. In turn, our adjustments will depend continuously on $w$. There are four cases to consider. To aid the reader we illustrate these respectively in Fig. \ref{deformglcurve} below. In each case, the dashed vertical line is $t=\rho_0$.
\begin{enumerate}
\item{} Suppose $\rho_{0}=\rho'_{0}=\rho''_{0}\leq \rho_{std}$. In this case, we perform a standard linear homotopy, along $(0, \rho_0]$, from $w$ to the torpedo function $f_{w(\rho_{0})}$. 
\item{} Next, suppose $\rho_{0}=\rho'_{0}\leq \rho_{std}$ and $\rho'_{0}<\rho''_{0}$. In this case, $w'$ reaches $0$ before $w''$. Near $\rho_0$, $w'$ is small. Thus, it follows from formula \ref{Rcurv}, that the function $w$ may be adjusted near $\rho_0$ and satisfy positivity of the scalar curvature, by keeping $w''$ non-positive. It is therefore possible to specify a parameter $\epsilon(w)$, depending continuously on $w$, and a homotopy of $w$ which is the identity off $(\rho_0- \epsilon(w), \rho_0+\epsilon(w))$, and so that the resulting function satisfies case 1. This homotopy involves continuously increasing the second derivative of $w$, so that it equals zero at $\rho_0$, but remains non-positive at all times. We then proceed as in case $1$.
\item{} Now, suppose $\rho_{0}=\rho''_{0}\leq \rho_{std}$ and $\rho''_{0}<\rho'_{0}$. Once again, we want to adjust $w$ so that it belongs in case 1. Straighten out the graph of $w$ immediately to the left of $\rho_{0}$, so that for some parameter $\epsilon(w)$, $w$ proceeds as a straight line segment with slope $w'(\rho_0)$ on $(\rho_{0}-\epsilon(w), \rho_0]$. This will necessitate homothetically shrinking the graph along of $w$ along $(0, \rho_{0}-\epsilon(w))$, but won't damage the positivity of the scalar curvature. Finally, on some subinterval $(\rho_{0}-\epsilon(w), \rho_{0}-\epsilon(w)+\epsilon')\subset (\rho_{0}-\epsilon(w), \rho_0]$, perform a gradual increase in the second derivative to satisfy the conditions of case 1. It follows from the bending argument of Gromov and Lawson, described in  detail in \cite{Walsh1}, that provided $w(\rho_{0}-\epsilon(w))$ is chosen small enough, this can be done without damaging positivity of the scalar curvature.
\item{} Finally, suppose $\rho_0=\rho_{std}<\min\{\rho'_{0}, \rho''_{0} \}$. In this case, $w$ is concave downward on all of $(0, \rho_{std}]$. Immediately to the left of $\rho_{std}$, adjust $w$ to bring the second derivative up to $0$ on some point $\rho_0-\epsilon(w)$. This can be done so that $w''\leq 0$ and so positive scalar curvature is maintained. Then proceed as in case 3. 
\end{enumerate}
\vspace{-1cm}
\begin{figure}[!htbp]
\hspace{-2cm}
\begin{picture}(0,0)%
\includegraphics{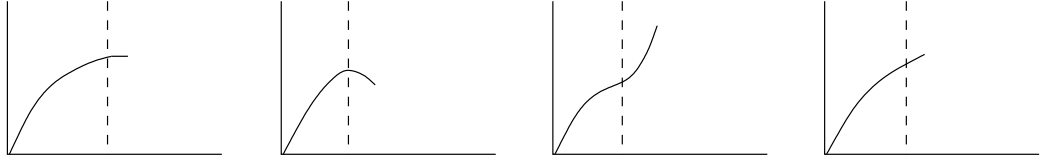}%
\end{picture}%
\setlength{\unitlength}{3947sp}%
\begingroup\makeatletter\ifx\SetFigFont\undefined%
\gdef\SetFigFont#1#2#3#4#5{%
  \reset@font\fontsize{#1}{#2pt}%
  \fontfamily{#3}\fontseries{#4}\fontshape{#5}%
  \selectfont}%
\fi\endgroup%
\begin{picture}(5079,1559)(1902,-7227)
\end{picture}%
\caption{The various forms $w$ may take near $0$.}
\label{deformglcurve}
\end{figure}   
\noindent Metrics which are already standard will immediately lie in case 1, in which case the corresponding linear homotopy will have no effect. This complete the proof.
\end{proof}

\section{The Main Theorem}
We are now in a position to prove our main theorem.

\begin{Theorem}\label{Chernyshthm}
Let $X$ be a smooth compact manifold of dimension $n$. Suppose $Y$ is obtained from $X$ by surgery on a sphere $i:S^{p}\hookrightarrow X$ with $p+q+1=n$ and  $p,q\geq 2$. Then the spaces $\Riem^{+}(X)$ and $\Riem^{+}(Y)$ are homotopy equivalent.
\end{Theorem}
\begin{proof}
Choose a reference metric $\m$ and a trivialisation $\phi:S^{p}\times \mathbb{R}^{q+1}\rightarrow \mathcal{N}$ as above.
Let $\bar{i}_{\bar{\rho}}:S^{p}\times D^{q+1}(\bar\rho)\hookrightarrow X$ be the framed embedding of the sphere $S^{p}$ described above, where $\bar{\rho}$ is chosen sufficiently small. The manifold $Y$ is assumed to be the manifold obtained by doing surgery on this embedding. This surgery can be canonically reversed by performing a surgery on the embedded $S^{q}$ of the attached handle. As $p,q\geq 2$, both surgeries are in codimension $\geq 3$.

Recall that $\Riem_{std}^{+}(X)$ denotes the space of psc-metrics on $X$ which are standard near $i(S^{p})$. Let $\Riem_{std}^{+}(Y)$ denote the analogous space of psc-metrics on $Y$, which are this time standard near the embedded $S^{q}\subset Y$. Note that the choice of reference metric and normal coordinates are unimportant here. Let $j: \Riem_{std}^{+}(X)\rightarrow \Riem_{std}^{+}(Y)$ be the map which is defined on a metric $g\in \Riem_{std}^{+}(X)$ as follows. Choose the smallest $\rho\in (0, \bar{\rho}]$ so that the metric induced on $N(\rho)$ is the torpedo $g_{tor}^{q+1}(\delta)$ on the disk factor. Thus, on each disk, the metric will only be infinitesimally a product near the boundary. Then attach to $g|_{X\setminus N(\rho)}$, a standard product $g_{tor}^{p+1}+\delta^{2} ds_{q}^{2}$, ensuring once again that $g_{tor}^{p+1}$ is again only infinitesimally a product near the boundary (the smallest possible torpedo metric). The resulting metric, denoted $j(g)$, is an element of $\Riem_{std}^{+}(Y)$. Furthermore, the association $g\mapsto j(g)$ defines a homeomorphism, with inverse defined completely analogously.

Applying Lemma \ref{Chetech} now gives us that the spaces ${\Riem}_{Astd}^{+}(X)$ and ${\Riem}_{Astd}^{+}(Y)$ are now homotopy equivalent and so it remains to show that ${\Riem}_{Astd}^{+}(X)$ is homotopy equivalent to ${\Riem^{+}(X)}$. It follows from the work of Palais in \cite{Palais} that the spaces ${\Riem}_{Astd}^{+}(X)$ and $\Riem^{+}(X)$ are dominated by CW-complexes. Thus, by the theorem of Whitehead, we need only show that the relative homotopy groups $\pi_{k}(\Riem^{+}(X),{\Riem}_{Astd}^{+}(X))$ are all trivial, in order to demonstrate the desired homotopy equivalence. 

Let $\lambda\in \pi_{k}(\Riem^{+}(X),{\Riem}_{Astd}^{+}(X))$. The element $\lambda$ is a homotopy equivalence class of commuting diagrams of the type shown below. 

\hspace{4cm}
\vspace{0.5cm}
\begin{tikzpicture}[description/.style={fill=white,inner sep=2pt}] 
\matrix (m) [matrix of math nodes, row sep=3em, 
column sep=2.5em, text height=1.5ex, text depth=0.25ex] 
{S^{k-1} & &  D^{k} \\ 
{\Riem}_{Astd}^{+}(X)  & & \Riem^{+}(X)\\ }; 
\path[->,font=\scriptsize] 
(m-1-1) edge node[auto] {} (m-2-1) 
(m-1-3) edge node[auto]  {}(m-2-3); 
\path[right hook->,font=\scriptsize] 
(m-1-1) edge node[auto]  {}(m-1-3)
(m-2-1) edge node[auto]  {}(m-2-3); 
\end{tikzpicture} 

\noindent Thus, $\lambda$ is represented by a family of psc-metrics in $\Riem^{+}(X)$ which is parameterised by a disk. Furthermore, the restriction of this parametrisation to the boundary $\p D^{k}=S^{k-1}$ is a map into ${\Riem}_{Astd}^{+}(X)$. In Theorem 2.13 of \cite{Walsh1}, we show that the Gromov-Lawson isotopy of Lemma \ref{GLisotopy}, holds for compact families of psc-metrics and so we may apply it to the family parameterised by the disk $D^{k}$. By definition, metrics in $\Riem_{Astd}^{+}(X)$ remain in this space throughout the isotopy, and so there is a continuous deformation through diagrams of the type above to one in which the image of the map from $D^{k}$ lies entirely in ${\Riem}_{Astd}^{+}(X)$. This completes the proof.
\end{proof}

An immediate corollary of Theorem \ref{Chernyshthm} is that, when $X$ is a simply connected spin manifold of dimension $\geq 5$, the homotopy type of the space $\Riem^{+}(X)$ is an invariant of spin cobordism. 
\begin{Corollary}\label{spinChernysh}
Let $X_0$ and $X_1$ be a pair of compact simply-connected spin manifolds of dimension $n\geq 5$. Suppose also that $X_0$ is spin cobordant to $X_1$. Then the spaces $\Riem^{+}(X_0)$ and $\Riem^{+}(X_1)$ are homotopy equivalent.
\end{Corollary}
\begin{proof}
This follows from Theorem B of \cite{GL}, where the authors show that any two such manifolds which are spin cobordant can be mutually obtained from the other by codimension $\geq 3$ surgeries. The idea is to show that, given some manifold $W$, a spin cobordism of $X_0$ and $X_1$, the interior of $W$ can be modified by surgery to make $H_1(W,X_0)$, $H_2(W,X_0)$, $H_n(W,X_0)$ and $H_{n-1}(W,X_0)$ trivial. By a theorem of Whitney, Theorem 2 of \cite{Whitney}, generators of these relative homology groups can be represented by embedded spheres. The fact that $W$ is spin (i.e. its first and second Stiefel-Whitney classes vanish) means that these embedded spheres have trivial normal bundle and so surgery can be performed. Thus, $W$ will admit a Morse function $f:W\rightarrow I$, with $f^{-1}(0)=X_0$, $f^{-1}(1)=X_1$ and so that all critical points of the functions $f$ and $1-f$ correspond to surgeries in codimension at least three. 
\end{proof}

\vspace{0.5cm}

\noindent {\em E-mail address:} walsmark@math.oregonstate.edu

\end{document}